\documentclass[11pt,a4paper]{article}

\usepackage{epsf,epsfig,amsfonts,amsgen,amsmath,amstext,amsbsy,amsopn,amsthm
}
\usepackage{amsmath}
\usepackage{amsfonts,amsthm,amssymb}
\usepackage{amsfonts}
\usepackage{graphics}
\usepackage{latexsym,bm}
\usepackage{amsfonts,amsthm,amssymb,bbding}
\usepackage{indentfirst}
\usepackage{graphicx}
\usepackage{color}
\usepackage[colorlinks=true,anchorcolor=blue,filecolor=blue,linkcolor=blue,urlcolor=blue,citecolor=blue]{hyperref}
\usepackage{float}
\usepackage{tikz}
\usepackage[mathscr]{eucal}
\newtheorem{thm}{Theorem}

\newtheorem{lemma}{Lemma}
\newtheorem{false statement}{False statement}

\theoremstyle{definition}

\newtheorem{claim}{Claim}

\newtheorem{ccclaim}{Claim}

\newtheorem{ccccclaim}{Claim}

\newtheorem{ccorollary}{Corollary}

\newtheorem{case}{Case}

\newtheorem{cccase}{Case}
\newtheorem{subcase}{Subcase}[case]

\begin{document}

\title{The index of $t\mathcal{{C}}_{3}^{-}$-free signed graphs
\footnote{Supported by Natural Science Foundation of Xinjiang Uygur Autonomous Region (No. 2024D01C41), Tianshan Talent Training Program (No. 2024TSYCCX0013), the Basic scientific research in universities of Xinjiang Uygur Autonomous Region (XJEDU2025P001) and NSFC (No. 12361071).}}
\author{{Dan Li}\thanks{Corresponding author. E-mail: ldxjedu@163.com.}, Mingsong Qin\\
{\footnotesize  College of Mathematics and System Science, Xinjiang University, Urumqi 830046, China}}
\date{}

\maketitle {\flushleft\large\bf Abstract:}
The classical spectral Tur\'{a}n problem is to determine the maximum spectral radius of an $F$-free graph of order $n$.  This paper extends this framework to signed graphs. Let  $\mathcal{C}_r^-$
be the set of all unbalanced signed graphs with underlying graphs $C_r$. Wang, Hou and Li [Linear Algebra Appl, 681 (2024) 47-65] previously determined the spectral Tur\'{a}n number of $\mathcal{C}_{3}^{-}$. In the present work, we characterize the extremal graphs that achieve the maximum index among all unbalanced signed graphs of order $n$ that are $t\mathcal{C}{3}^{-}$-free for $t\geq 2$. Furthermore, for $t\geq 3$, we identify the graphs with the second maximum index among all $t\mathcal{C}{3}^{-}$-free unbalanced signed graphs of fixed order $n$.

\vspace{0.1cm}
\begin{flushleft}
\end{flushleft}

\section{Introduction}
Let $G$ be a simple graph with vertex set $V(G)=\{v_1,...,v_n\}$ and edge set $E(G)=\{e_1,...,e_m\}$. The order and size of $G$ are defined as $|V(G)|$ and $|E(G)|$, respectively. Given a graph \( F \), a graph is said to be \( F \)-free if it does not contain a subgraph isomorphic to \( F \). The Tur\'{a}n number of \( F \), denoted by \( ex(n,F) \), is the maximum number of edges in an \( n \)-vertex \( F \)-free graph. An \( F \)-free graph is said to be extremal with respect to \( ex(n,F) \) if it has \( ex(n,F) \) edges. Denote by \( T_{n,r} \) the complete \( r \)-partite graph on \( n \) vertices in which all parts are as equal in size as possible. In 1941, Tur\'{a}n \cite{14} determined that the \(r\)-partite Tur\'{a}n graph \(T_{n,r}\) is the unique \(n\)-vertex graph maximizing the number of edges (Tur\'{a}n number \(ex(n,K_{r+1})\)) among all \(K_{r+1}\)-free graphs. Since then, the Tur\'{a}n problem has been extended to various forbidden subgraphs, including disjoint unions of classic graph structures. For instance, Moon \cite{121} and Simonovits \cite{S2} independently determined that for sufficiently large \(n\), the join of a complete graph and an \(r\)-partite Tur\'{a}n graph is the unique extremal graph for \(k K_{r+1}\)-free graphs (graphs without \(k\) vertex-disjoint copies of \(K_{r+1}\)). In 1962, Erd\H{o}s \cite{71} studied the Tur\'{a}n number for \(t C_3\)-free graphs (no \(t\) disjoint triangles) for $n>400(t-1)^2$, and subsequent works generalized this to disjoint odd cycles \(t C_{2\ell+1}\) \cite{72}.

Let \( A(G) \) be the adjacency matrix of a graph \( G \), and \( \rho(G) \) be its spectral radius. The spectral extremal value of a given graph \( F \), denoted by \( \text{spec}(n, F) \), is the maximum spectral radius over all \( n \)-vertex \( F \)-free graphs. An \( F \)-free graph on \( n \) vertices with maximum spectral radius is called an extremal graph with respect to \( \text{spec}(n, F) \). With the development of spectral graph theory, Nikiforov \cite{N1} showed that the \(r\)-partite Tur\'{a}n graph \(T_{n,r}\) also maximizes the spectral radius among \(K_{r+1}\)-free graphs. For disjoint subgraphs, Ni, Wang and Kang \cite{N2} proved that the spectral extremal graph for \(k K_{r+1}\)-free graphs is the same join structure as the edge-extremal graph, while Fang, Zhai and Lin \cite{72} obtained the extremal spectral radius \( \text{spec}(n, tC_l) \) for any fixed $t$, $l$ and large enough $n$. 

 An underlying graph $G$ together with a signature $\sigma:E(G)\to\{-1,+1\}$  forms a signed graph $\Gamma=(G,\sigma)$. In a signed graph, edge signs are usually interpreted as $\pm1$. An edge $e$ is positive (resp. negative) if $\sigma(e)=+1$ (resp. $\sigma(e)=-1$). A cycle in $\Gamma$ is said to be positive if it contains an even number of negative edges, otherwise it is negative. $\Gamma=(G,\sigma)$ is balanced if it has no negative cycles, otherwise it is unbalanced. Let $U\subset V(G)$, the operation of reversing the signs of all edges between $U$ and $V(G)\setminus U$ is called a switching operation. If a signed graph $\Gamma^\prime$ is obtained from $\Gamma$ by applying finitely many switching operations, then $\Gamma$ is said to be switching equivalent to $\Gamma^\prime$. For more details about the notion of signed graphs, we refer to \cite{1}. Signed graphs were first introduced in works of Harary \cite{8} and Cartwright and Harary \cite{5}, and the matroids of graphs were extended to matroids of signed graphs by Zaslavsky \cite{19}. Chaiken \cite{6} and Zaslavsky \cite{19} independently established the Matrix-Tree Theorem for signed graph . The theory of signed graphs is a special case of gain graphs and biased graphs \cite{20}. The adjacency matrix of $\Gamma$ is defined as $A(\Gamma)=(a_{ij}^\sigma)$, where $a^\sigma_{ij}=\sigma(v_iv_j)$ if $v_i\sim v_j $, otherwise, $a^\sigma_{ij}=0$. The eigenvalues of $\Gamma$ are written as $\lambda_1(A(\Gamma))\geq\lambda_{2}(A(\Gamma))\geq\cdots\geq\lambda_n(A(\Gamma))$  in decreasing  order which are the eigenvalues of $A(\Gamma)$ and $\lambda_1(A(\Gamma))$ is the index of $\Gamma$. The index has been extensively studied in the literature, with relevant works including \cite{12,3,9}

\begin{figure}
	\centering
	\ifpdf
	\setlength{\unitlength}{1bp}%
	\begin{picture}(428.31, 150.05)(0,0)
		\put(0,0){\includegraphics{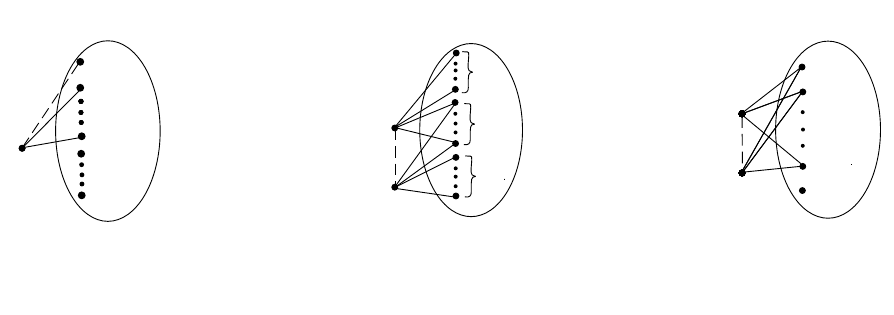}}
		\put(40.65,105.72){\fontsize{11.38}{13.66}\selectfont $v_2$}
		\put(41.68,84.11){\fontsize{11.38}{13.66}\selectfont $v_{t-1}$}
		\put(41.33,73.82){\fontsize{11.38}{13.66}\selectfont $v_t$}
		\put(41.68,54.11){\fontsize{11.38}{13.66}\selectfont $v_{n-1}$}
		\put(40.82,118.58){\fontsize{11.38}{13.66}\selectfont $v_1$}
		\put(1.67,78.57){\fontsize{11.38}{13.66}\selectfont $u$}
		\put(38.02,34.01){\fontsize{11.38}{13.66}\selectfont $\Gamma_{n,t}$}
		\put(16.12,98.53){\fontsize{11.38}{13.66}\selectfont $-$}
		\put(37.13,135.49){\fontsize{11.38}{13.66}\selectfont $K_{n-1}$}
		\put(174.54,87.74){\fontsize{11.38}{13.66}\selectfont $v_1$}
		\put(178.53,74.70){\fontsize{11.38}{13.66}\selectfont $-$}
		\put(212.30,133.68){\fontsize{11.38}{13.66}\selectfont $K_{n-2}$}
		\put(210.80,34.77){\fontsize{11.38}{13.66}\selectfont $\Sigma_{s,t,r}$}
		\put(175.71,57.93){\fontsize{11.38}{13.66}\selectfont $v_2$}
		\put(133.21,8.12){\fontsize{11.38}{13.66}\selectfont Fig.1. \label{Fig1} The signed graphs $\Gamma_{n,t}, \Sigma_{s,t,r}, U_1$.}
		\put(229.95,89.37){\fontsize{11.38}{13.66}\selectfont $t$}
		\put(228.34,113.10){\fontsize{11.38}{13.66}\selectfont $s$}
		\put(231.07,63.32){\fontsize{11.38}{13.66}\selectfont $r$}
		\put(388.85,68.07){\fontsize{11.38}{13.66}\selectfont $v_{n-1}$}
		\put(388.21,115.48){\fontsize{11.38}{13.66}\selectfont $v_3$}
		\put(379.75,133.61){\fontsize{11.38}{13.66}\selectfont $K_{n-2}$}
		\put(388.47,103.18){\fontsize{11.38}{13.66}\selectfont $v_4$}
		\put(387.22,34.80){\fontsize{11.38}{13.66}\selectfont $U_1$}
		\put(388.66,57.24){\fontsize{11.38}{13.66}\selectfont $v_n$}
		\put(342.32,65.28){\fontsize{11.38}{13.66}\selectfont $v_2$}
		\put(345.75,80.49){\fontsize{11.38}{13.66}\selectfont $-$}
		\put(340.71,92.85){\fontsize{11.38}{13.66}\selectfont $v_1$}
	\end{picture}%
	\else
	\setlength{\unitlength}{1bp}%
	\begin{picture}(428.31, 150.05)(0,0)
			\put(0,0){\includegraphics{111.pdf}}
		\put(40.65,105.72){\fontsize{11.38}{13.66}\selectfont $v_2$}
		\put(41.68,84.11){\fontsize{11.38}{13.66}\selectfont $v_{t-1}$}
		\put(41.33,73.82){\fontsize{11.38}{13.66}\selectfont $v_t$}
		\put(41.68,54.11){\fontsize{11.38}{13.66}\selectfont $v_{n-1}$}
		\put(40.82,118.58){\fontsize{11.38}{13.66}\selectfont $v_1$}
		\put(1.67,78.57){\fontsize{11.38}{13.66}\selectfont $u$}
		\put(38.02,34.01){\fontsize{11.38}{13.66}\selectfont $\Gamma_{n,t}$}
		\put(16.12,98.53){\fontsize{11.38}{13.66}\selectfont $-$}
		\put(37.13,135.49){\fontsize{11.38}{13.66}\selectfont $K_{n-1}$}
		\put(174.54,87.74){\fontsize{11.38}{13.66}\selectfont $v_1$}
		\put(178.53,74.70){\fontsize{11.38}{13.66}\selectfont $-$}
		\put(212.30,133.68){\fontsize{11.38}{13.66}\selectfont $K_{n-2}$}
		\put(210.80,34.77){\fontsize{11.38}{13.66}\selectfont $\Sigma_{s,t,r}$}
		\put(175.71,57.93){\fontsize{11.38}{13.66}\selectfont $v_2$}
		\put(133.21,8.12){\fontsize{11.38}{13.66}\selectfont Fig.1. \label{Fig1} The signed graphs $\Gamma_{n,t}, \Sigma_{s,t,r}, U_1$.}
		\put(229.95,89.37){\fontsize{11.38}{13.66}\selectfont $t$}
		\put(228.34,113.10){\fontsize{11.38}{13.66}\selectfont $s$}
		\put(231.07,63.32){\fontsize{11.38}{13.66}\selectfont $r$}
		\put(388.85,68.07){\fontsize{11.38}{13.66}\selectfont $v_{n-1}$}
		\put(388.21,115.48){\fontsize{11.38}{13.66}\selectfont $v_3$}
		\put(379.75,133.61){\fontsize{11.38}{13.66}\selectfont $K_{n-2}$}
		\put(388.47,103.18){\fontsize{11.38}{13.66}\selectfont $v_4$}
		\put(387.22,34.80){\fontsize{11.38}{13.66}\selectfont $U_1$}
		\put(388.66,57.24){\fontsize{11.38}{13.66}\selectfont $v_n$}
		\put(342.32,65.28){\fontsize{11.38}{13.66}\selectfont $v_2$}
		\put(345.75,80.49){\fontsize{11.38}{13.66}\selectfont $-$}
		\put(340.71,92.85){\fontsize{11.38}{13.66}\selectfont $v_1$}
	\end{picture}%
	\fi
\end{figure}

In recent years, the study of spectral Tur\'{a}n-type problems has been expanded from simple undirected graphs to signed graphs. Given a set $\mathcal{F}$ of signed graphs, if a signed graph $\Gamma$ contains no signed subgraph isomorphic to any one in $\mathcal{F}$, then $\Gamma$ is called $\mathcal{F}$-free. Different from aforementioned studies, we focus on signed graphs, and ask what are the maximum spectral radius or index of an $\mathcal{F}$-free signed graph of order $n$. Let $\mathcal{K}_r^-$ and $\mathcal{C}_r^-$ be the sets of all unbalanced signed graphs with underlying graphs $K_r$ and $C_r$, respectively. Chen and Yuan \cite{7} and Wang \cite{k1}  gave the spectral Tur\'{a}n number of $\mathcal{K}_{4}^{-}$ and $\mathcal{K}_{5}^{-}$, respectively, while Xiong and Hou \cite{k2} gave the maximum index for $\mathcal{K}_r^-$--free unbalanced signed graphs. For cycle-related forbidden subgraphs in signed graphs, Wang, Hou and Li \cite{15} determined the spectral Tur\'{a}n number of $\mathcal{C}_{3}^{-}$. The $\mathcal{C}_{4}^{-}$-free unbalanced signed graphs of fixed order with maximum index have been determined by Wang and Lin \cite{16}. Moreover, Wang, Hou and Huang \cite{18} gave the spectral Tur\'{a}n number of $\mathcal{C}_{2k+1}^{-}$, where $3 \leq k \leq n/10-1$. Let $\Gamma=(K_n,H^-)$ be a signed complete graph with a negative edge-induced subgraph $H$. For unbalanced connected signed graphs with $n\geq3$ vertices, Brunetti and Stani\'{c} \cite{4} showed that a signed graph maximizes the index if and only if it is switching isomorphic to $(K_n,K_2^-)$. Let $t\mathcal{C}_{3}^-$ be the set of $t$ unbalanced $C_3$. It is particularly noted that graph $(K_n,K_2^-)$ is the extremal signed graph with the maximum index in the class of graphs that forbid $t$ vertex-disjoint unbalanced $C_3$. Motivated by these works, this paper continues the investigation into the characterization of extremal graphs with the maximum index in the class of $t\mathcal{C}_{3}^-$-free unbalanced signed graph with $t\geq 2$. Notably, we no longer impose the requirement that these $t$ unbalanced triangles are vertex-disjoint. In Fig.\ref {Fig1}, we use dashed lines to represent negative edges and solid lines to represent positive edges. Let $K_{n-1}$ be a complete graph with vertex set $\{v_1,...,v_{n-1}\}$.  $\Gamma_{n,t}$ is a signed graph obtained by appending a new vertex $u$ to $K_{n-1}$ and joining $u$ to $t-1$ vertices $v_1,...,v_{t-1}$, with $uv_1$ being the unique negative edge. The main results of this paper are as follows.

\begin{thm}\label{thm1}
 Let $\Gamma$ be a $t\mathcal{C}_{3}^{-}$-free unbalanced signed graph of order $n$ with maximum index $(t\geq 2, n\geq 6)$.  Then $$\Gamma\cong \begin{cases}\Gamma_{n,t+1}&\text{if}~ 2\leq t\leq n-2,\\\Gamma_{n,n}&\text{if}~ t\geq n-1.\end{cases}$$ 
\end{thm}

\begin{thm}\label{thm2}
	For integers $n,t$ with  $t\geq 3$ and $n\geq 9$. Let $\Gamma$ be a  $t\mathcal{C}_{3}^{-}$-free unbalanced signed graph of order $n$ with maximum index such that $\Gamma$ is not switching isomorphic to $\Gamma_{n,t+1}$ for $3\leq t\leq n-2$ or $\Gamma_{n,n}$ for $t\geq n-1$. Then
	
	(i) If $3\leq t\leq\left\lfloor\frac{n}{2}\right\rfloor$, then $\Gamma\cong \Gamma_{n,t}$,
	
	(ii) If $\left\lfloor\frac{n}{2}\right\rfloor+1\leq t\leq n-3$, then $\Gamma\cong \Sigma_{1,t-1,n-t-2}$,
	
	(iii) If $t=n-2$, then $\Gamma\cong \Gamma_{n,n-2}$,
	
	(iv) If $t\geq n-1$, then  $\Gamma\cong \Gamma_{n,n-1}$.
\end{thm}

\section{Preliminaries}\label{se2}
Let $M$ be a real symmetric matrix with block form $M= [M_{ij}]$, and $q_{ij}$ denote the average row sum of $M_{ij}$. The matrix $Q=(q_{ij})$ is the quotient matrix of $M$. Furthermore, $Q$ is referred to as an equitable quotient matrix if every block $M_{ij}$ has a constant row sum. Let Spec$(Q)=\{\lambda_1^{[t_1]},...,\lambda_k^{[t_k]}\}$ represent the spectrum of $Q$, where 
$\lambda_i$ is an eigenvalue with multiplicity $t_i$ for $1 \leq i \leq k$. Let $P_{Q}(\lambda)=det(\lambda I-Q)$ denote the characteristic polynomial of $Q$.
\begin{lemma}\label{l}\cite{21}
	There are two kinds of eigenvalues of the real symmetric matrix $M$.

	(i) The eigenvalues that match the eigenvalues of $Q$.
	
	(ii) The eigenvalues of $M$ not in Spec$(Q)$  that are unchanged when $\alpha $J is added to block $M_{ij}$  for every $1\leq i,j \leq m$, where  $\alpha$ is any constant. Moreover, $\lambda_1(M)=\lambda_1(Q)$ when $M$ is irreducible and nonnegative.
\end{lemma}
\begin{lemma}\label{l1} \cite{k2}

	(i) $\lambda_1(A(\Gamma_{n,t}))$ is the largest root of $g_{n,t}(\lambda)=0$, where
	\begin{center}
		$g_{n,t}(\lambda)=\lambda^3-(n-3)\lambda^2-(n+t-3)\lambda-t^2+(n+4)t-n-7$.
	\end{center}

	(ii) $n-2\leq\lambda_1(A(\Gamma_{n,t}))<n-1$, with left equality if and only if $t=3$.
\end{lemma}
 The matrix $J_{r\times s}$ is an all-one matrix of size $r \times s$, and when $r=s$ it is denoted by $J_r$. Also, we use $j_k=(1,\ldots,1)^T\in R^k$.
\begin{lemma}\label{lq1}
	Let $n\geq 9$ and $t$ be  positive integers with $3\leq t\leq n-3$.  Let $\Gamma_{n,t}$ and $\Sigma_{1,t-1,n-t-2}$ be the signed graphs depicted in Fig.\ref{Fig1}. Then 
	$$ \begin{cases}\lambda_1(A(\Gamma_{n,t}))>\lambda_1(A(\Sigma_{1,t-1,n-t-2}))&\text{if}~ 3\leq t\leq\left\lfloor\frac{n}{2}\right\rfloor,\\\lambda_1(A(\Sigma_{1,t-1,n-t-2}))>\lambda_1(A(\Gamma_{n,t}))&\text{if}~ \left\lfloor\frac{n}{2}\right\rfloor+1\leq t\leq n-3.\end{cases}$$ 
\end{lemma}
\begin{proof}
Without loss of generality, assume that $v_4$ is a vertex in the complete graph $K_{n-2}$ that is not adjacent to $v_2$ in $\Sigma_{1,t-1,n-t-2}$. We first present $A(\Sigma_{1,t-1,n-t-2})$ and its corresponding quotient matrix $Q_1$ based on the vertex partition $V_1=\{v_1\}$, $V_2=\{v_2\}$, $V_3=\{v_4\}$, $V_4=\{v_3,v_5,...,v_{t+2}\}$ and $V_5=\{v_{t+3},...,v_n\}$ as follows
\begin{center}
 $\begin{gathered} A(\Sigma_{1,t-1,n-t-2})= \begin{bmatrix} 0 & -1 & 1 & j_{t-1}^T & \bf{0}_{n-t-2}^T \\ -1 & 0 & 0 & j_{t-1}^T & j_{n-t-2}^T \\ 1 & 0 & 0 & j_{t-1}^T & j_{n-t-2}^T \\ j_{t-1} & j_{t-1} & j_{t-1} & \left(J-I\right)_{t-1} & J_{(t-1)\times(n-t-2)} \\ \bf{0}_{n-t-2} & j_{n-t-2} & j_{n-t-2} & J_{(n-t-2)\times(t-1)} & \left(J-I\right)_{n-t-2} \end{bmatrix} \end{gathered}$
\end{center}
and
\begin{center}
$Q_1= \begin{bmatrix} 0 & -1 & 1 & t-1 & 0 \\ -1 & 0 & 0 & t-1 & n-t-2 \\ 1 & 0 & 0 & t-1 & n-t-2 \\ 1 & 1 & 1 & t-2 & n-t-2 \\ 0 & 1 & 1 & t-1 & n-t-3 \end{bmatrix}.$
\end{center}
 Note that the characteristic polynomial of $Q_1$ is $P_{Q_1}(\lambda)=\lambda^5+(5-n)\lambda^4+(9-3n-t)\lambda^3+(nt-n-t^2-2t-1)\lambda^2+(2nt+4n-2t^2-2t-16)\lambda+4n-12$. Adding $\alpha $J to the blocks of $A(\Sigma_{1,t-1,n-t-2})$, where $\alpha $ is constant, then
\begin{center}
$A_1= \begin{bmatrix} 0 & 0 & 0 & \bf{0}^T & \bf{0}^T \\ 0 & 0 & 0 & \bf{0}^T & \bf{0}^T \\ 0 & 0 & 0 & \bf{0}^T & \bf{0}^T \\ \bf{0} & \bf{0} & \bf{0} & -I_{t-1} & \bf{0}^T \\ \bf{0} & \bf{0} & \bf{0} & \bf{0} & -I_{n-t-2} \end{bmatrix}.$
\end{center}
 Since $\lambda_1(Q_1)>0$ and Spec$(A_1)$$=\begin{Bmatrix}-1^{[n-3]},0^{[3]}\end{Bmatrix}$,  $\lambda_1(A(\Sigma_{1,t-1,n-t-2}))=\lambda_1(Q_1)$. Note that 
 \begin{align*}
 P_{Q_1}^{(1)}(n-3)&=n^4-13n^3+(61-t)n^2+(-2t^2+10t-127)n+4t^2-17t+98\\&\geq n^4-16n^3+76n^2-132n+185,\\
 	P_{Q_1}^{(2)}(n-3)&=8n^3-66n^2+(178-4t)n-2t^2+14t-164\\&\geq8n^3-66n^2+(178-4(n-3))n-2(n-3)^2+14\times3-164\\&=8n^3-72n^2+202n-140>0,\\
 	P_{Q_1}^{(3)}(n-3)&=36n^2-186n-6t+234\\&\geq36n^2-186n-6(n-3)+234\\&=6(6n^2-32n+42)>0,\\
 P_{Q_1}^{(4)}(n-3)&=96n-240>0.
 \end{align*}
 Since  $P_{Q_1}^{(4-i)}(\lambda)$ is strictly increasing for $\lambda\geq n-3$ as $P_{Q_1}^{(5-i)}(n-3)>0$ for $i=1,2,3$. Let $h(n)=n^4-16n^3+76n^2-132n+185$. Note that 
 \begin{align*}
 &h^{(1)}(n)=4n^3-48n^2+152n-132,\\
 &h^{(2)}(n)=12n^2-96n+152,\\
 &h^{(3)}(n)=24n-96.
 \end{align*}
 Obviously, $h^{(3)}(n)>0$ for $n\geq 9$. Thus, $h^{(2)}(n)\geq h^{(2)}(9)=260>0$ and $h^{(1)}(n)\geq h^{(1)}(9)=264>0$. Hence, $h(n)$ is a monotone increasing for $n\geq 9$ and $h(n)\geq h(9)=50>0$. This implies that $P_{Q_1}^{(1)}(n-3)>0$ and $P_{Q_1}(\lambda)$ is  strictly increasing for $\lambda\geq n-3$. By Lemma \ref{l1},  $\lambda_1(A(\Gamma_{n,t}))$ is the largest root of $g_{n,t}(\lambda)=0$. Let $f(\lambda)=(\lambda+1)^2g_{n,t}(\lambda)$, then 
 \begin{center}
 	$f(\lambda)-P_{Q_1}(\lambda)=\lambda^3+(3+4t-3n)\lambda^2+(5+9t-7n)\lambda+(t-5)(n-t-1)$.
 \end{center}
 Set $\lambda_1=\lambda_1(A(\Gamma_{n,t}))$, we can obtain that
 \begin{center}
 	$0-P_{Q_1}(\lambda_1)=\lambda_1^3+(3+4t-3n)\lambda_1^2+(5+9t-7n)\lambda_1+(t-5)(n-t-1)$.
 \end{center}
For $\left\lfloor\frac{n}{2}\right\rfloor+1\leq t\leq n-3$, we have
 \begin{align*}
	-P_{Q_1}(\lambda_1)&\geq(n-2)^3+(3+4t-3n)(n-2)^2+(5+9t-7n)(n-2)+(t-5)(n-t-1)\\&=-2n^3+(2+4t)n^2-(2-6t)n-(t-1)^2\\&\geq-2n^3+(2+4(\lfloor\frac{n}{2}\rfloor+1))n^2-(2-6(\lfloor\frac{n}{2}\rfloor+1))n-(n-4)^2\\&=8n^2+12n-16>0.
 \end{align*}
 Thus, $P_{Q_1}(\lambda_1)<0$. Since $\lambda_1>n-3$ and $P_{Q_1}(\lambda)$ is  strictly increasing for $\lambda\geq n-3$, we have 	$\lambda_1(A(\Sigma_{1,t-1,n-t-2}))>\lambda_1(A(\Gamma_{n,t}))$ for $\left\lfloor\frac{n}{2}\right\rfloor+1\leq t\leq n-3$.
 For  $3\leq t\leq\left\lfloor\frac{n}{2}\right\rfloor$, similarly, 
 \begin{align*}
 	-P_{Q_1}(\lambda_1)&<(n-1)^3+(3+4t-3n)(n-1)^2+(5+9t-7n)(n-1)+(t-5)(n-t-1)\\&=-2n^3+(4t-1)n^2+(2t+1)n-t^2-t+2<0.
 \end{align*}
 Hence, $P_{Q_1}(\lambda_1)>0$. Since $\lambda_1>n-3$ and $P_{Q_1}(\lambda)$ is  strictly increasing for $\lambda\geq n-3$, we have 	$\lambda_1(A(\Sigma_{1,t-1,n-t-2}))<\lambda_1(A(\Gamma_{n,t}))$ 
for $3\leq t\leq\left\lfloor\frac{n}{2}\right\rfloor$. The proof is completed.
\end{proof}
\begin{lemma}\label{lqq1}
	Let $n\geq 9$  be a positive integer. Let $\Gamma_{n,n-2}$ and $U_1$ be the graphs depicted in Fig.\ref {Fig1}. Then 
\begin{center}
$\lambda_1(A(\Gamma_{n,n-2}))>\lambda_1(A(U_1))$.
\end{center}
\end{lemma}
\begin{proof}
Without loss of generality, assume that $v_n$ is a vertex in the complete graph $K_{n-2}$ that is not adjacent to $v_1$ and $v_2$ in $U_1$.  We first present $A(U_1)$ and its corresponding quotient matrix $Q_2$ using the vertex partition $V_1=\{v_1\}$, $V_2=\{v_2\}$,  $V_3=\{v_3,...,v_{n-1}\}$ and $V_4={v_n}$ as follows\\
\begin{center}
$\begin{gathered} A(U_1)= \begin{bmatrix} 0 & -1 & j_{n-3}^T & 0 \\ -1 & 0 & j_{n-3}^T & 0 \\ j_{n-3} & j_{n-3} & (J-I)_{n-3} & j_{n-3}^T \\ 0 & 0 & j_{n-3} & 0 \end{bmatrix} \end{gathered}$
\end{center}
and
\begin{center}
	$Q_2= \begin{bmatrix} 0 & -1 & n-3 & 0 \\ -1 & 0 & n-3 & 0 \\ 1 & 1 & n-4 & 1 \\ 0 & 0 & n-3 & 0 \end{bmatrix}$.
\end{center}
 Note that the characteristic polynomial of $Q_2$ is $P_{Q_2}(\lambda)=\lambda^4+(4-n)\lambda^3+(8-3n)\lambda^2+(3n-10)\lambda+n-3$. Adding $\alpha $J to the blocks of $A(U_1)$, where $\alpha $ is a constant, then
 \begin{center}
 $A_2= \begin{bmatrix} 0 & 0 & \bf{0}^T & 0 \\ 0 & 0 & \bf{0}^T & 0 \\ \bf{0} & \bf{0} & -I_{n-3} & \bf{0}^T \\ 0 & 0 & \bf{0} & 0 \end{bmatrix}$.
 \end{center}
Since $\lambda_1(Q_2)>0$ and Spec$(A_2)$$=\begin{Bmatrix}-1^{[n-3]},0^{[3]}\end{Bmatrix}$,  $\lambda_1(A(U_1))=\lambda_1(Q_2)$. Note that 
\begin{center}
$(\lambda+1)g_{n,n-2}(\lambda)-P_{Q_2}(\lambda)=4(n-\lambda-4)$.
\end{center}
Let $\lambda_1$ be the largest root of $P_{Q_2}(\lambda)=0$. If $\lambda_1\leq n-4$, then $\lambda_1(A(\Gamma_{n,n-2}))>\lambda_1(A(U_1))$ by Lemma \ref{l1}. If $\lambda_1> n-4$, then $(\lambda_1+1)g_{n,n-2}(\lambda_1)=4(n-\lambda_1-4)<0$. Thus, $g_{n,n-2}(\lambda_1)<0$. This means that $\lambda_1(A(\Gamma_{n,n-2}))>\lambda_1(A(U_1))$. The proof is completed.
\end{proof}

\section{Proofs of Theorems \ref{thm1} and \ref{thm2}}
Let $\Gamma$ be a signed graph. The degree of a vertex $v_i$ in $\Gamma$ is denoted by $d_\Gamma(v_i)$ which is the number of edges incident with $v_i$.  We denote the set of all neighbors of $u$ in $\Gamma$ by $N_\Gamma(u)$ and $N_\Gamma[u]=N_\Gamma(u)\cup\{u\}$. Let $\rho(\Gamma)=\max\{|\lambda_i(\Gamma)|{:}1\leq i\leq n\}$ be the spectral radius of $\Gamma$. For $\phi\neq U\subset V(\Gamma)$, let $\Gamma[U]$ be the signed subgraph of $\Gamma$ induced by $U$. We denote by $\Gamma+uv$ the signed graph obtained from $\Gamma$ by adding the positive edge $uv$ and by $\Gamma-uv$ the signed graph obtained from  $\Gamma$ by deleting the edge $uv$, where $u,v\in V(\Gamma)$. If all edges of $K_n$ are positive, then  we denote the graph by $(K_n,+)$.

\begin{lemma}\label{l2}\cite{S1}
	Let $\Gamma$ be a signed graph. Then there exists a signed graph $\Gamma^{\prime}$ switching equivalent to $\Gamma$ such that $A(\Gamma^\prime)$ has a non-negative eigenvector corresponding to $\lambda_1(A(\Gamma^{\prime}))$.
\end{lemma}
\begin{lemma}\label{x1}\cite{15} 
Let $\Gamma=(G,\sigma)$ be a connected unbalanced signed graph of order $n$. If $\Gamma$ is $\mathcal{C}^-_3$-free, then $\rho(\Gamma)\leq \frac{1}{2}(\sqrt{n^2-8}+n-4)$.
\end{lemma}
\begin{lemma}\label{l5}\cite{B1}
Two signed graphs with the same underlying graph are switching equivalent if and only if they have the same set of positive cycles.
\end{lemma}
The subsequent lemma acts as a key instrument in this paper.
\begin{lemma}\label{l6}
	Let $\Gamma$ be a $t\mathcal{C}_{3}^{-}$-free unbalanced signed graph of order $n$ with $2\leq t\leq n-2$ and $n\geq 6$. Then $\lambda_1(A(\Gamma))\leq \lambda_1(A(\Gamma_{n,t+1}))$, with equality if and only if $\Gamma$ is switching isomorphic to $\Gamma_{n,t+1}$. 
\end{lemma}

\begin{proof}
 Let $\Gamma=(G,\sigma)$ be a  $t\mathcal{C}_{3}^{-}$-free unbalanced signed graph on $n$ vertices with maximum index. According to Lemma \ref{l2}, $\Gamma$ is switching equivalent to a signed graph $\Gamma^{\prime}$ such that $A(\Gamma^\prime)$ has a non-negative eigenvector corresponding to $\lambda_1(A(\Gamma^\prime))=\lambda_1(A(\Gamma))$. Then by Lemma \ref{l5}, $\Gamma^\prime$ is also unbalanced, $t\mathcal{C}_{3}^{-}$-free and it attains the maximum index among all $t\mathcal{C}_{3}^{-}$-free unbalanced signed graphs. Let $V(\Gamma^\prime)=\{v_1,v_2,...,v_n\}$ and $X=(x_1,x_2,...,x_n)^T$ be the non-negative unit eigenvector of $A(\Gamma^\prime)$ corresponding to $\lambda_1(A(\Gamma^\prime))$. Note that $\Gamma_{n,3}$ is unbalanced and $t\mathcal{C}_{3}^{-}$-free. By Lemma \ref{l1}, $\lambda_1(A(\Gamma^\prime))\geq \lambda_1(A(\Gamma_{n,3}))=n-2$. Since $\frac{1}{2}(\sqrt{n^2-8}+n-4) <n-2$, Lemma \ref{x1} guarantees that $\Gamma^\prime$ must contain an unbalanced $C_3$ as a signed subgraph. Assume that  $V(C_3)=\{v_1,v_2,v_3\}$.
\begin{claim}\label{c1}
$X$ contains at most one zero entry.
\end{claim}
\begin{proof}
Suppose for contradiction that $X$ contains at least two zero entries, say $x_n=x_{n-1}=0$, then 
\begin{align*}
\lambda_1(A(\Gamma^{\prime}))&=X^TA(T^{\prime})X=(x_1,\ldots,x_{n-2})A(\Gamma^{\prime}-v_n-v_{n-1})(x_1,\ldots,x_{n-2})^T\\&\leq\lambda_1(A(\Gamma^{\prime}-v_n-v_{n-1}))\leq\lambda_1(A(K_{n-2}))=n-3<\lambda_1(A(\Gamma^{\prime})),
\end{align*}
a contradiction. Thus, $X$ contains at most one zero entry.
\end{proof} 
\begin{claim}\label{c2}
The unbalanced $C_3$ contains all negative edges of $\Gamma^\prime$.
\end{claim}
\begin{proof}
Otherwise, suppose that there is a negative edge $e=v_iv_j\notin E(C_3)$. Let $\Gamma^{\prime\prime}=\Gamma^{\prime}-e$, then $\Gamma^{\prime\prime}$ is a $t\mathcal{C}_{3}^{-}$-free unbalanced signed graph and
\begin{align*}
	\lambda_1(A(\Gamma^{\prime\prime}))-\lambda_1(A(\Gamma^{\prime}))&\geq X^T(A(\Gamma^{\prime\prime})-A(\Gamma^{\prime}))X
	\\&=2x_ix_j\geq0.
\end{align*} 
If $\lambda_1(A(\Gamma^{\prime\prime}))=\lambda_1(A(\Gamma^{\prime}))$, then $X$ is also an eigenvector of $A(\Gamma^{\prime\prime})$ corresponding to $\lambda_1(A(\Gamma^{\prime\prime}))$. Based on the following equations,\\

$\lambda_1(A(\Gamma^{\prime}))x_i=\sum\limits_{v_s\in N_{\Gamma^\prime}(v_i)}\sigma^{\prime}(v_sv_i)x_s$,\\

$\lambda_1(A(\Gamma^{\prime}))x_j=\sum\limits_{v_s\in N_{\Gamma^\prime}(v_j)}\sigma^{\prime}(v_sv_j)x_s$,\\

$\lambda_1(A(\Gamma^{\prime\prime}))x_i=\sum\limits_{v_s\in N_{\Gamma^\prime}(v_i)}\sigma^{\prime}(v_sv_i)x_s+x_j$\\
and\\

$\lambda_1(A(\Gamma^{\prime\prime}))x_j=\sum\limits_{v_s\in N_{\Gamma^\prime}(v_j)}\sigma^{\prime}(v_sv_j)x_s+x_i$,\\
we obtain that $x_i=x_j=0$, which contradicts Claim \ref{c1}. Thus, $\lambda_1(A(\Gamma^{\prime\prime}))>\lambda_1(A(\Gamma^{\prime}))$, a contradiction.
\end{proof}
 Assume that $k$ is the smallest positive integer such that $x_{k}=\max_{1\leq i\leq n}x_{i}$. By Claim \ref{c1}, $x_k>0$ clearly.
\begin{claim}\label{o3}
The unbalanced $C_3$ contains exactly one negative edge.
\end{claim}
\begin{proof}
Otherwise, the unbalanced $C_3$ contains three negative edges of $\Gamma^\prime$.  Recall that there is at most one zero entry of $X$ by Claim \ref{c1}. If $k\leq 3$, then
\begin{align*}
\lambda_1(A(\Gamma^{\prime}))x_k&=-(x_1+x_2+x_3)+x_k+\sum_{v_i\in N_{\Gamma^\prime}(v_k)\setminus V(C_3)}x_i\\&\leq-(x_1+x_2+x_3)+x_k+(n-3)x_k\\&<(n-3)x_k.
\end{align*}
This implies that $\lambda_1(A(\Gamma^{\prime}))< n-3$, a contradiction. Thus, $k>4$. And then
\begin{center}
$(n-2)x_k\leq\lambda_1(A(\Gamma^{\prime}))x_k=\sum\limits_{v_i\in N_{\Gamma^\prime}(v_k)}x_i\leq d_{\Gamma^{\prime}}(v_k)x_k$,
\end{center}
that is, $d_{\Gamma^{\prime}}(v_k)=n-2$ or $n-1$. Assume that $d_{\Gamma^{\prime}}(v_k)=n-2$, then $x_i=x_k$ for any $v_i\in N_{\Gamma^{\prime}}(v_k)$. If $t=2$, then $\Gamma^{\prime}$ contains $2\mathcal{{C}}_3^-$, a contradiction. If $3\leq t\leq n-2$, it means that at least one of $x_1, x_2,x_3$ equals $x_k$, contradicting the choice of $k$. Thus, $d_{\Gamma^{\prime}}(v_k)=n-1$. If $2\leq t \leq3$, then $\Gamma^{\prime}$ contains $t\mathcal{{C}}_3^-$, a contradiction. For $4\leq t\leq n-2$, let $\Gamma^{\prime\prime}=\Gamma^{\prime}-v_1v_2$, then $\Gamma^{\prime\prime}$ is a $t\mathcal{C}_{3}^{-}$-free unbalanced signed graph and
\begin{align*}
	\lambda_1(A(\Gamma^{\prime\prime}))-\lambda_1(A(\Gamma^{\prime}))&\geq X^T(A(\Gamma^{\prime\prime})-A(\Gamma^{\prime}))X
	\\&=2x_1x_2\geq0.
\end{align*} 
If $\lambda_1(A(\Gamma^{\prime\prime}))=\lambda_1(A(\Gamma^{\prime}))$, then $X$ is also an eigenvector of $A(\Gamma^{\prime\prime})$ corresponding to $\lambda_1(A(\Gamma^{\prime\prime}))$. From the eigenvector equations,\\

$\lambda_1(A(\Gamma^{\prime}))x_1=\sum\limits_{v_s\in N_{\Gamma^\prime}(v_1)}\sigma^{\prime}(v_sv_1)x_s$,\\

$\lambda_1(A(\Gamma^{\prime}))x_2=\sum\limits_{v_s\in N_{\Gamma^\prime}(v_2)}\sigma^{\prime}(v_sv_2)x_s$,\\

$\lambda_1(A(\Gamma^{\prime\prime}))x_1=\sum\limits_{v_s\in N_{\Gamma^\prime}(v_1)}\sigma^{\prime}(v_sv_1)x_s+x_2$\\
and\\

$\lambda_1(A(\Gamma^{\prime\prime}))x_2=\sum\limits_{v_s\in N_{\Gamma^\prime}(v_2)}\sigma^{\prime}(v_sv_2)x_s+x_1$,\\
we obtain that $x_1=x_2=0$, which contradicts Claim \ref{c1}. Thus, $\lambda_1(A(\Gamma^{\prime\prime}))>\lambda_1(A(\Gamma^{\prime}))$, a contradiction. So, the unbalanced $C_3$ contains exactly one negative edge.
\end{proof}
 Claims \ref{c2} and \ref{o3} show that $\Gamma^{\prime}$ contains only one negative edge, and it is the negative edge of the unbalanced $C_3$. Assume that this edge is $v_1v_2$.
\begin{claim}\label{CC1}
If $X>0$, then $k \geq 3$ and $d_{\Gamma^{\prime}}(v_k)=n-1$.
\end{claim}
\begin{proof}
If $k<3$, then $(n-2)x_k\leq\lambda_1(A(\Gamma^{\prime}))x_k\leq-x_{3-k}+(n-2)x_k<(n-2)x_k$, a contradiction. Thus, $k \geq 3$. Note that
\begin{center}
$(n-2)x_k\leq\lambda_1(A(\Gamma^{\prime}))x_k=\sum\limits_{v_i\in N_{\Gamma^\prime}(v_k)}x_i\leq d_{\Gamma^{\prime}}(v_k)x_k$,
\end{center}
then $d_{\Gamma^{\prime}}(v_k) \geq n-2$. If $d_{\Gamma^{\prime}}(v_k)=n-2$, then $x_i=x_k$ for all $v_i\in N_{\Gamma^{\prime}}v_k)$, meaning at least one of $x_1, x_2$ equals $x_k$, contradicting the choice of $k$. Hence,  $d_{\Gamma^{\prime}}(v_k)=n-1$.
\end{proof}
Next, we divide the proof into the following two cases.
\begin{case}\label{Case1}
There exists an integer $r\in\{1,2,\ldots,n\}$ such that $x_r=0$.
\end{case}
Firstly, we assert that $d_{\Gamma^{\prime}}(v_r)\geq 1$. Otherwise, let $\Gamma^{\prime\prime}=\Gamma^{\prime}+v_1v_r$, then $\Gamma^{\prime\prime}$ is a $t\mathcal{C}_{3}^{-}$-free unbalanced signed graph and
\begin{align*}
	\lambda_1(A(\Gamma^{\prime\prime}))-\lambda_1(A(\Gamma^{\prime}))&\geq X^T(A(\Gamma^{\prime\prime})-A(\Gamma^{\prime}))X
	\\&=2x_1x_r\geq0.
\end{align*} 
If $\lambda_1(A(\Gamma^{\prime\prime}))=\lambda_1(A(\Gamma^{\prime}))$, then $X$ is also an eigenvector of $A(\Gamma^{\prime\prime})$ corresponding to $\lambda_1(A(\Gamma^{\prime\prime}))$. Based on the following equations,\\

$\lambda_1(A(\Gamma^{\prime}))x_1=\sum\limits_{v_s\in N_{\Gamma^\prime}(v_1)}\sigma^{\prime}(v_sv_1)x_s$,\\

$\lambda_1(A(\Gamma^{\prime}))x_r=\sum\limits_{v_s\in N_{\Gamma^\prime}(v_r)}\sigma^{\prime}(v_sv_r)x_s$,\\

$\lambda_1(A(\Gamma^{\prime\prime}))x_1=\sum\limits_{v_s\in N_{\Gamma^\prime}(v_1)}\sigma^{\prime}(v_sv_1)x_s+x_r$\\
and\\

$\lambda_1(A(\Gamma^{\prime\prime}))x_r=\sum\limits_{v_s\in N_{\Gamma^\prime}(v_r)}\sigma^{\prime}(v_sv_r)x_s+x_1$,\\
we obtain that $x_1=x_r=0$, which contradicts Claim \ref{c1}. Thus, $\lambda_1(A(\Gamma^{\prime\prime}))>\lambda_1(A(\Gamma^{\prime}))$, a contradiction. So, $d_{\Gamma^{\prime}}(v_r)\geq 1$. If $r\geq 3$, then $0=\lambda_1(A(\Gamma^{\prime}))x_r=\sum_{v_i\in N_{\Gamma^{\prime}}(v_r)}x_i>0$, a contradiction. Thus, $r=1$ or $2$. Without loss of generality, assume that $r=1$. Then  $k\geq 2$. Note that
\begin{center}
	$(n-2)x_k\leq\lambda_1(A(\Gamma^{\prime}))x_k=\sum\limits_{v_i\in N_{\Gamma^\prime}(v_k)}x_i\leq d_{\Gamma^{\prime}}(v_k)x_k$,
\end{center}
then  $d_{\Gamma^{\prime}}(v_k) \geq n-2$. If  $d_{\Gamma^{\prime}}(v_k)=n-2$, then each of the $n-2$ entries  of $X$ corresponding to the neighbors of $v_k$ is equal to $x_k$. It means that $x_2=\cdots=x_n$ since $x_1=0$. If $d_{\Gamma^{\prime}}(v_k)=n-1$, then
\begin{center}
$(n-2)x_{{k}}\leq\lambda_{1}(A(\Gamma^{\prime}))x_{{k}}=x_{1}+\sum\limits_{v_{i}\in N_{\Gamma^{\prime}}(v_{k})\setminus\{v_{1}\}}x_{i}\leq(d_{\Gamma^{\prime}}(v_{{k}})-1)x_{{k}}=(n-2)x_k$.
\end{center}
Equality again forces $x_2=\cdots=x_n$. In either case, $d_{\Gamma^{\prime}}(v_i)=n-2$ or $n-1$ and $v_i$ is adjacent to all other vertices $V(\Gamma^{\prime})\backslash\{v_1\}$ for any $i\in[2,n]$. Therefore, $\Gamma^{\prime}[V(\Gamma^{\prime})\backslash\{v_1\}]\cong(K_{n-1},+)$. Note that $\Gamma$ is a $t\mathcal{C}_{3}^{-}$-free unbalanced signed graph for $2\leq t\leq n-2$, then we assert that there exist exactly $t$ vertices of degree $n-1$. To verify, let $r(r\neq t)$ denote the number of vertices with degree $n-1$ . If $r>t$, then $\Gamma^{\prime}$ contains $t\mathcal{{C}}_3^-$, a contradiction. If $r\leq t-1$, assume that $d_{\Gamma^{\prime}}(v_i)=n-1$ for $i\in[2,r+1]$, let $\Gamma^{\prime\prime}= \Gamma^{\prime}+v_1v_{r+2}$, then $\Gamma^{\prime\prime}$ is a $t\mathcal{C}_{3}^{-}$-free unbalanced signed graph and
\begin{align*}
	\lambda_1(A(\Gamma^{\prime\prime}))-\lambda_1(A(\Gamma^{\prime}))&\geq X^T(A(\Gamma^{\prime\prime})-A(\Gamma^{\prime}))X
	\\&=2x_1x_{r+2}\geq0.
\end{align*} 
If $\lambda_1(A(\Gamma^{\prime\prime}))=\lambda_1(A(\Gamma^{\prime}))$, then $X$ is also an eigenvector of $A(\Gamma^{\prime\prime})$ corresponding to $\lambda_1(A(\Gamma^{\prime\prime}))$. Based on the following equations,\\

$\lambda_1(A(\Gamma^{\prime}))x_1=\sum\limits_{v_s\in N_{\Gamma^\prime}(v_1)}\sigma^{\prime}(v_sv_1)x_s$,\\

$\lambda_1(A(\Gamma^{\prime}))x_{r+2}=\sum\limits_{v_s\in N_{\Gamma^\prime}(v_{r+2})}\sigma^{\prime}(v_sv_{r+2})x_s$,\\

$\lambda_1(A(\Gamma^{\prime\prime}))x_1=\sum\limits_{v_s\in N_{\Gamma^\prime}(v_1)}\sigma^{\prime}(v_sv_1)x_s+x_{r+2}$\\
and\\

$\lambda_1(A(\Gamma^{\prime}))x_{r+2}=\sum\limits_{v_s\in N_{\Gamma^\prime}(v_{r+2})}\sigma^{\prime}(v_sv_{r+2})x_s+x_1$,\\
we obtain that $x_1=x_{r+2}=0$, which contradicts Claim \ref{c1}. Thus,  $\lambda_1(A(\Gamma^{\prime\prime}))>\lambda_1(A(\Gamma^{\prime}))$, a contradiction.
Thus, there exist $t$ vertices with a degree of $n-1$ and $\Gamma$ is switching isomorphic to $\Gamma_{n,t+1}$. 
\begin{case}\label{Case2}
$X>0$, i.e., all entries of $X>0$ are positive.
\end{case}
By Claim \ref{CC1}, $k \geq 3$ and $d_{\Gamma^{\prime}}(v_k)=n-1$. Without loss of generality, we suppose that $0<x_1\leq x_2$. Next, we will further discuss in two subcases.
\begin{subcase}
$t=2$
\end{subcase}
Obviously, $k=3$. Otherwise,  if $k\geq 4$, then $\Gamma^\prime$ contains two distinct unbalanced $\mathcal{C}_{3}^{-}$ subgraphs $v_1v_2v_3v_1$ and $v_1v_2v_kv_1$, a contradiction. Thus, $k=3$. This means that $N_{\Gamma^{\prime}}(v_1)\cap N_{\Gamma^{\prime}}(v_2)=\{v_3\}$. Firstly, we assert that $\Gamma^{\prime}[V(\Gamma^{\prime})\backslash\{v_1,v_2\}]\cong(K_{n-2},+)$. Suppose for contradiction there exists an edge $uv\notin \Gamma^{\prime}[V(\Gamma^{\prime})\backslash\{v_1,v_2\}]$, let $\Gamma^{\prime\prime}= \Gamma^{\prime}+uv$, then $\Gamma^{\prime\prime}$ is a $2\mathcal{C}_{3}^{-}$-free unbalanced signed graph and
\begin{align*}
	\lambda_1(A(\Gamma^{\prime\prime}))-\lambda_1(A(\Gamma^{\prime}))&\geq X^T(A(\Gamma^{\prime\prime})-A(\Gamma^{\prime}))X
	\\&=2x_ux_{v}>0,
\end{align*} 
a contradiction. Thus, $\Gamma^{\prime}[V(\Gamma^{\prime})\backslash\{v_1,v_2\}]\cong(K_{n-2},+)$. Next, we claim that either $v_iv_1\in E(\Gamma^{\prime})$ or $v_iv_2\in E(\Gamma^{\prime})$ for any $i\in [4,n]$. Otherwise, assume that $v_sv_1,v_sv_2\notin E(\Gamma^{\prime})$, let $\Gamma^{\prime\prime}= \Gamma^{\prime}+v_2v_s$, then $\Gamma^{\prime\prime}$ is a $2\mathcal{C}_{3}^{-}$-free unbalanced signed graph and
\begin{align*}
	\lambda_1(A(\Gamma^{\prime\prime}))-\lambda_1(A(\Gamma^{\prime}))&\geq X^T(A(\Gamma^{\prime\prime})-A(\Gamma^{\prime}))X
	\\&=2x_2x_{s}>0,
\end{align*} 
a contradiction. Finally, we assert that $d_{\Gamma^{\prime}}(v_1)=2$. Otherwise, $d_{\Gamma^{\prime}}(v_1)\geq 3$, then $\Gamma^\prime$ is switching isomorphic to $\Sigma_{s,1,r}$, where $s\geq 1$ and $s+r=n-3$. Let $\Gamma^{\prime\prime}= \Gamma^{\prime}-v_1v_i+v_2v_i$ for any $v_i\in  N_{\Gamma^{\prime}}(v_1)\backslash\{v_{2},v_{3}\}$, then  $\Gamma^{\prime\prime}$ is a $2\mathcal{C}_{3}^{-}$-free unbalanced signed graph and
\begin{align*}
	\lambda_1(A(\Gamma^{\prime\prime}))-\lambda_1(A(\Gamma^{\prime}))&\geq X^T(A(\Gamma^{\prime\prime})-A(\Gamma^{\prime}))X
	\\&=2\sum\limits_{v_i\in N_{\Gamma^{\prime}}(v_1)\backslash\{v_{2},v_{3}\}}x_i(x_2-x_1)\geq0.
\end{align*} 
If $\lambda_1(A(\Gamma^{\prime\prime}))=\lambda_1(A(\Gamma^{\prime}))$, then $X$ is also an eigenvector of $A(\Gamma^{\prime\prime})$ corresponding to $\lambda_1(A(\Gamma^{\prime\prime}))$. Based on the following equations,\\

$\lambda_1(A(\Gamma^{\prime}))x_1=-x_2+x_3+\sum\limits_{v_i\in N_{\Gamma^{\prime}}(v_1)\backslash\{v_{2},v_{3}\}}x_i$\\
and\\

$\lambda_1(A(\Gamma^{\prime\prime}))x_1=-x_2+x_3$,\\
we obtain that $\sum\limits_{v_i\in N_{\Gamma^{\prime}}(v_1)\backslash\{v_{2},v_{3}\}}x_i=0$, contradicting $X>0$. Thus, $\lambda_1(A(\Gamma^{\prime\prime}))>\lambda_1(A(\Gamma^{\prime}))$, a contradiction. Thus, $d_{\Gamma^{\prime}}(v_1)=2$ and $\Gamma^{\prime}$ is switching isomorphic to $\Gamma_{n,3}$. 
\begin{subcase}
 $3\leq t\leq n-2$.
\end{subcase}
Firstly, we assert that $\Gamma^{\prime}[V(\Gamma^{\prime})\backslash\{v_1,v_2\}]\cong(K_{n-2},+)$. Otherwise, assume that there exists an edge $uv\notin \Gamma^{\prime}[V(\Gamma^{\prime})\backslash\{v_1,v_2\}]$, let $\Gamma^{\prime\prime}= \Gamma^{\prime}+uv$, then $\Gamma^{\prime\prime}$ is a $t\mathcal{C}_{3}^{-}$-free unbalanced signed graph and
\begin{align*}
	\lambda_1(A(\Gamma^{\prime\prime}))-\lambda_1(A(\Gamma^{\prime}))&\geq X^T(A(\Gamma^{\prime\prime})-A(\Gamma^{\prime}))X
	\\&=2x_ux_{v}>0,
\end{align*} 
a contradiction. Thus, $\Gamma^{\prime}[V(\Gamma^{\prime})\backslash\{v_1,v_2\}]\cong(K_{n-2},+)$. Next, we assert that $|N_{\Gamma^{\prime}}(v_1)\cap N_{\Gamma^{\prime}}(v_2)|=t-1$. Otherwise, $|N_{\Gamma^{\prime}}(v_1)\cap N_{\Gamma^{\prime}}(v_2)|\leq t-2$, which implies the existence of a vertex $v_s\notin {N_{\Gamma^{\prime}}(v_1)\cap N_{\Gamma^{\prime}}(v_2)}$. If $v_s\notin N_{\Gamma^{\prime}}(v_1)$ and $v_s\notin N_{\Gamma^{\prime}}(v_2)$ simultaneously, let $\Gamma^{\prime\prime}= \Gamma^{\prime}+v_1v_s+v_2v_s$, then $\Gamma^{\prime\prime}$ is a $t\mathcal{C}_{3}^{-}$-free unbalanced signed graph and
\begin{align*}
	\lambda_1(A(\Gamma^{\prime\prime}))-\lambda_1(A(\Gamma^{\prime}))&\geq X^T(A(\Gamma^{\prime\prime})-A(\Gamma^{\prime}))X
	\\&=2(x_1+x_2)x_s>0,
\end{align*} 
a contradiction. For the cases where $v_s\in {N_{\Gamma^{\prime}}(v_1)\setminus N_{\Gamma^{\prime}}(v_2)}$ or $v_s\in {N_{\Gamma^{\prime}}(v_2)\setminus N_{\Gamma^{\prime}}(v_1)}$, analogous contradictions are uniformly derived in both scenarios. Thus, $|N_{\Gamma^{\prime}}(v_1)\cap N_{\Gamma^{\prime}}(v_2)|=t-1$. Now, we claim that either $v_iv_1\in E(\Gamma^{\prime})$ or $v_iv_2\in E(\Gamma^{\prime})$ for any $v_i\in V(\Gamma^{\prime})\backslash(N_{\Gamma^{\prime}}[v_1]\cap N_{\Gamma^{\prime}}[v_2])$. Suppose for contradiction $v_sv_1\notin E(\Gamma^{\prime})$ and  $v_sv_2\notin E(\Gamma^{\prime})$, let $\Gamma^{\prime\prime}= \Gamma^{\prime}+v_2v_s$, then $\Gamma^{\prime\prime}$ is a $t\mathcal{C}_{3}^{-}$-free unbalanced signed graph and
\begin{align*}
	\lambda_1(A(\Gamma^{\prime\prime}))-\lambda_1(A(\Gamma^{\prime}))&\geq X^T(A(\Gamma^{\prime\prime})-A(\Gamma^{\prime}))X
	\\&=2x_2x_{s}>0,
\end{align*} 
a contradiction. Finally, we assert that $d_{\Gamma^{\prime}}(v_1)=t$. Otherwise, $d_{\Gamma^{\prime}}(v_1)\geq t+1$, then $\Gamma^\prime$ is switching isomorphic to $\Sigma_{s,t-1,r}$, where $s\geq 1$ and $s+r=n-t-1$. Let $\Gamma^{\prime\prime}= \Gamma^{\prime}-v_1v_i+v_2v_i$ for any $v_i\in  N_{\Gamma^{\prime}}[v_1]\backslash(N_{\Gamma^{\prime}}[v_1]\cap N_{\Gamma^{\prime}}[v_2])$, then  $\Gamma^{\prime\prime}$ is a $t\mathcal{C}_{3}^{-}$-free unbalanced signed graph and
\begin{align*}
	\lambda_1(A(\Gamma^{\prime\prime}))-\lambda_1(A(\Gamma^{\prime}))&\geq X^T(A(\Gamma^{\prime\prime})-A(\Gamma^{\prime}))X
	\\&=2\sum\limits_{v_i\in N_{\Gamma^{\prime}}[v_1]\backslash(N_{\Gamma^{\prime}}[v_1]\cap N_{\Gamma^{\prime}}[v_2])}x_i(x_2-x_1)\geq0.
\end{align*} 
If $\lambda_1(A(\Gamma^{\prime\prime}))=\lambda_1(A(\Gamma^{\prime}))$, then $X$ is also an eigenvector of $A(\Gamma^{\prime\prime})$ corresponding to $\lambda_1(A(\Gamma^{\prime\prime}))$. Based on the following equations,\\

$\lambda_1(A(\Gamma^{\prime}))x_1=\sum\limits_{v_s\in N_{\Gamma^{\prime}}(v_1)}x_s$\\
and\\

$\lambda_1(A(\Gamma^{\prime\prime}))x_1=\sum\limits_{v_s\in N_{\Gamma^{\prime}}(v_1)}x_s-\sum\limits_{v_i\in N_{\Gamma^{\prime}}[v_1]\backslash(N_{\Gamma^{\prime}}[v_1]\cap N_{\Gamma^{\prime}}[v_2])}x_i$,\\
we obtain that $\sum\limits_{v_i\in N_{\Gamma^{\prime}}[v_1]\backslash(N_{\Gamma^{\prime}}[v_1]\cap N_{\Gamma^{\prime}}[v_2])}x_i=0$, which contradicts Case \ref{Case2}. Thus, $\lambda_1(A(\Gamma^{\prime\prime}))>\lambda_1(A(\Gamma^{\prime}))$, a contradiction. Thus, $d_{\Gamma^{\prime}}(v_1)=t$ and $\Gamma^{\prime}$ is switching isomorphic to $\Gamma_{n,t+1}$. The proof is completed.
\end{proof}

\noindent{\bf{Proof of Theorem \ref{thm1}}.} By Lemma \ref{l6},  we prove that if $2\leq t\leq n-2$, then $\Gamma\cong \Gamma_{n,t+1}$. The remaining case $t\geq n-1$ requires proving $\Gamma\cong \Gamma_{n,n}$.  Let $\Gamma=(G,\sigma)$ be a  $t\mathcal{C}_{3}^{-}$-free unbalanced signed graph on $n$ vertices with maximum index for $t\geq 2,n\geq 6$. According to Lemma \ref{l2}, $\Gamma$ is switching equivalent to a signed graph $\Gamma^{\prime}$ such that $A(\Gamma^\prime)$ has a non-negative eigenvector corresponding to $\lambda_1(A(\Gamma^\prime))=\lambda_1(A(\Gamma))$. Then  $\Gamma^\prime$ is also unbalanced and $t\mathcal{C}_{3}^{-}$-free by Lemma \ref{l5}. Furthermore, $\Gamma^\prime$ also has the maximum index among all $t\mathcal{C}_{3}^{-}$-free unbalanced signed graphs. Let $V(\Gamma^\prime)=\{v_1,v_2,...,v_n\}$ and $X=(x_1,x_2,...,x_n)^T$ be the non-negative unit eigenvector of $A(\Gamma^\prime)$ corresponding to $\lambda_1(A(\Gamma^\prime))$. Note that $\Gamma_{n,3}$ is unbalanced and $t\mathcal{C}_{3}^{-}$-free. By Lemma \ref{l1}, $\lambda_1(A(\Gamma^\prime))\geq \lambda_1(A(\Gamma_{n,3}))=n-2$. Since $\frac{1}{2}(\sqrt{n^2-8}+n-4) <n-2$, $\Gamma^\prime$ must contain an unbalanced $C_3$ as a signed subgraph by Lemma \ref{x1}. Assume that $C_3$ is an unbalanced signed subgraph of $\Gamma^{\prime}$ and $V(C_3)=\{v_1,v_2,v_3\}$.

By similar arguments as in the proof of Lemma \ref{l6}, $X$ contains at most one zero entry and $\Gamma^\prime$ has exactly one negative edge, which lies in the unbalanced $C_3$. Assume that this edge is $v_1v_2$. Then it is obvious that $\Gamma^\prime$ contains at most $(n-2)\mathcal{C}_3^-$. Since $t\geq n-1$, $\Gamma^\prime$ is $t\mathcal{C}_{3}^{-}$-free.  Lemma \ref{l6} implies the index sequence $\lambda_1(A(\Gamma_{n,3}))< \lambda_1(A(\Gamma_{n,4}))<\dots < \lambda_1(A(\Gamma_{n,n}))$. Thus, $\Gamma\cong \Gamma_{n,n}$. So, the proof is completed.

Let $\mathcal{F}_t$ and $\mathcal{B}_t$ represent the family of friendship signed graphs and book signed graphs, respectively, each family consists of signed graphs containing $t$ $(t\geq 2)$ unbalanced $C_3$.
\begin{ccorollary}\label{cco1}
Let $\Gamma=(G,\sigma)$ be a unbalanced signed graph of order $n$. If $\Gamma$ is $\mathcal{F}_t$-free, then $\lambda_1(A(\Gamma))\leq \lambda_1(A(\Gamma_{n,n})) $.
\end{ccorollary}
\begin{proof}
 Let $\Gamma=(G,\sigma)$ be a  $\mathcal{F}_t$-free unbalanced signed graph on $n$ vertices with maximum index for $t\geq 2$ and $n\geq 6$. According to Lemma \ref{l2}, $\Gamma$ is switching equivalent to a signed graph $\Gamma^{\prime}$ such that $A(\Gamma^\prime)$ has a non-negative eigenvector corresponding to $\lambda_1(A(\Gamma^\prime))=\lambda_1(A(\Gamma))$. Note that $\Gamma_{n,3}$ is  $\mathcal{F}_t$-free. By Lemma \ref{l1}, $\lambda_1(A(\Gamma^\prime))\geq \lambda_1(A(\Gamma_{n,3}))=n-2$. Since $\frac{1}{2}(\sqrt{n^2-8}+n-4) <n-2$, $\Gamma^\prime$ must contain an unbalanced $C_3$ as a signed subgraph by Lemma \ref{x1}. Assume that $C_3$ is such an unbalanced signed subgraph of $\Gamma^{\prime}$ with $V(C_3)=\{v_1,v_2,v_3\}$. By similar arguments as in the proof of Lemma \ref{l6}, $X$ contains at most one zero entry and the unbalanced $C_3$ contains all negative edges of $\Gamma^\prime$.

\begin{ccclaim}\label{ooou3}
	The unbalanced $C_3$ contains exactly one negative edge.
\end{ccclaim}
\begin{proof}
Suppose for contradiction that the unbalanced $C_3$ contains three negative edges of $\Gamma^\prime$.  Recall that there is at most one zero entry of $X$. If $k\leq 3$, then
\begin{align*}
	\lambda_1(A(\Gamma^{\prime}))x_k&=-(x_1+x_2+x_3)+x_k+\sum_{v_i\in N_{\Gamma^\prime}(v_k)\setminus V(C_3)}x_i\\&\leq-(x_1+x_2+x_3)+x_k+(n-3)x_k\\&<(n-3)x_k,
\end{align*}
i.e., $\lambda_1(A(\Gamma^{\prime}))< n-3$, a contradiction. Thus, $k>4$. And then
\begin{center}
	$(n-2)x_k\leq\lambda_1(A(\Gamma^{\prime}))x_k=\sum\limits_{v_i\in N_{\Gamma^\prime}(v_k)}x_i\leq d_{\Gamma^{\prime}}(v_k)x_k$,
\end{center}
that is, $d_{\Gamma^{\prime}}(v_k)=n-2$ or $n-1$. If $d_{\Gamma^{\prime}}(v_k)=n-2$, then $x_i=x_k$ for any $v_i\in N_{\Gamma^{\prime}}(v_k)$. It means that at least one of $x_1, x_2, x_3$ equals $x_k$, contradicting the choice of $k$. Thus, $d_{\Gamma^{\prime}}(v_k)=n-1$. Let $\Gamma^{\prime\prime}=\Gamma^{\prime}-v_1v_2$, then $\Gamma^{\prime\prime}$ is  still a  $\mathcal{F}_t$-free unbalanced signed graph but $\lambda_1(A(\Gamma^{\prime\prime}))>\lambda_1(A(\Gamma^{\prime}))$. This contradicts the maximality of $\lambda_1(A(\Gamma^\prime))$. So, the unbalanced $C_3$ contains exactly one negative edge.
\end{proof}
Claim \ref{ooou3} show that $\Gamma^{\prime}$ contains only one negative edge, and it is the negative edge of the unbalanced $C_3$. Assume that this edge is $v_1v_2$.
\begin{ccclaim}
$\Gamma^\prime\cong \Gamma_{n,n}$.
\end{ccclaim}
Otherwise $\Gamma^\prime\ncong \Gamma_{n,n}$. It means that there exists an edge  $uv\notin E(\Gamma^{\prime})$,
let $\Gamma^{\prime\prime}= \Gamma^{\prime}+uv$, then $\Gamma^{\prime\prime}$ is a  $\mathcal{F}_t$-free unbalanced signed graph and
\begin{align*}
	\lambda_1(A(\Gamma^{\prime\prime}))-\lambda_1(A(\Gamma^{\prime}))&\geq X^T(A(\Gamma^{\prime\prime})-A(\Gamma^{\prime}))X
	\\&=2x_ux_{v}\geq 0,
\end{align*} 
If $\lambda_1(A(\Gamma^{\prime\prime}))=\lambda_1(A(\Gamma^{\prime}))$, then $X$ is also an eigenvector of $A(\Gamma^{\prime\prime})$ corresponding to $\lambda_1(A(\Gamma^{\prime\prime}))$. Consider the following equations,\\

$\lambda_1(A(\Gamma^{\prime}))x_u=\sum\limits_{v_s\in N_{\Gamma^\prime}(v_u)}\sigma^{\prime}(v_sv_u)x_s$,\\

$\lambda_1(A(\Gamma^{\prime}))x_v=\sum\limits_{v_s\in N_{\Gamma^\prime}(v_v)}\sigma^{\prime}(v_sv_v)x_s$,\\

$\lambda_1(A(\Gamma^{\prime\prime}))x_u=\sum\limits_{v_s\in N_{\Gamma^\prime}(v_u)}\sigma^{\prime}(v_sv_u)x_s+x_v$\\
and\\

$\lambda_1(A(\Gamma^{\prime\prime}))x_v=\sum\limits_{v_s\in N_{\Gamma^\prime}(v_v)}\sigma^{\prime}(v_sv_v)x_s+x_u$,\\
we obtain that $x_u=x_v=0$, which contradicts Claim \ref{c1}. Thus, $\lambda_1(A(\Gamma^{\prime\prime}))>\lambda_1(A(\Gamma^{\prime}))$, a contradiction. So, $\Gamma^\prime\cong \Gamma_{n,n}$. The proof is completed.
\end{proof}
Through proof analogous to those in Corollary \ref{cco1} and the  Case \ref{Case2} of Lemma \ref{l6}, the following corollary holds directly.
\begin{ccorollary}
	Let $\Gamma=(G,\sigma)$ be a unbalanced signed graph of order $n$. If $\mathcal{B}_t$-free, then
 $$\Gamma\cong \begin{cases}\Gamma_{n,t+1}&\text{if}~ 2\leq t\leq n-2,\\\Gamma_{n,n}&\text{if}~ t\geq n-1.\end{cases}$$. 
\end{ccorollary}

\noindent{\bf{Proof of Theorem \ref{thm2}}.}  Let $\Gamma$ be a  signed graph with maximum index among $t\mathcal{C}_{3}^{-}$-free unbalanced signed graphs on $n$ vertices such that $\Gamma$ is not switching isomorphic to $\Gamma_{n,t+1}$  for $3\leq t\leq n-2$ or $\Gamma_{n,n}$ for $t\geq n-1$. According to Lemma \ref{l2}, $\Gamma$ is switching equivalent to a signed graph $\Gamma^{\prime}$ such that $A(\Gamma^\prime)$ has a non-negative eigenvector corresponding to $\lambda_1(A(\Gamma^\prime))=\lambda_1(A(\Gamma))$. Lemma \ref{l5} then implies $\Gamma^\prime$ is also unbalanced, $t\mathcal{C}_{3}^{-}$-free and has the maximum index among all $t\mathcal{C}_{3}^{-}$-free unbalanced signed graphs. Let $V(\Gamma^\prime)=\{v_1,v_2,...,v_n\}$ and $X=(x_1,x_2,...,x_n)^T$ be the non-negative unit eigenvector of $A(\Gamma^\prime)$ corresponding to $\lambda_1(A(\Gamma^\prime))$. By Lemma \ref{l1}, $\lambda_1(A(\Gamma_{n,t})\geq n-2$ for $t\geq 3$, with equality if and only if $t=3$. Then $\lambda_1(A(\Gamma^\prime))\geq \lambda_1(A(\Gamma_{n,t}))=n-2$. Since $\frac{1}{2}(\sqrt{n^2-8}+n-4) <n-2$, $\Gamma^\prime$ must contain an unbalanced $C_3$ as a signed subgraph by Lemma \ref{x1}. Assume that  $V(C_3)=\{v_1,v_2,v_3\}$.

By similar arguments as in the proof of Lemma \ref{l6}, we establish the following claims.
\begin{ccccclaim}\label{cccccla1}
	$X$ contains at most one zero entry.
\end{ccccclaim}
\begin{ccccclaim}\label{cla2}
	The unbalanced $C_3$ contains all negative edges of $\Gamma^\prime$.
\end{ccccclaim}
\begin{ccccclaim}\label{ola3}
	The unbalanced $C_3$ contains exactly one negative edge.
\end{ccccclaim}
\begin{ccccclaim}\label{CC1}
	If $X>0$, then $k \geq 3$ and $d_{\Gamma^{\prime}}(v_k)=n-1$.
\end{ccccclaim}
Claims \ref{cla2} and \ref{ola3} together imply $\Gamma^{\prime}$ contains exactly one negative edge, which lies in the unbalanced $C_3$. Assume that this edge is $v_1v_2$. Assume that $k$ is the smallest positive integer such that $x_{k}=\max_{1\leq i\leq n}x_{i}$, and $x_2\geq x_1$. According to Claim \ref{cccccla1}, we consider the following two cases.
\begin{cccase}\label{Case1}
	There exists an integer $r$ such that $x_r=0$ for $1\leq r \leq n$.
\end{cccase}
First, we consider the case where $3\leq t \leq n-2$. We assert that $d_{\Gamma^{\prime}}(v_r)\geq 1$. Otherwise, $d_{\Gamma^{\prime}}(v_r)=0$. Let $\Gamma^{\prime\prime}=\Gamma^{\prime}+v_1v_r$, then $\Gamma^{\prime\prime}$ is a $t\mathcal{C}_{3}^{-}$-free unbalanced signed graph, not switching isomorphic to $\Gamma_{n,t+1}$ and
\begin{align*}
	\lambda_1(A(\Gamma^{\prime\prime}))-\lambda_1(A(\Gamma^{\prime}))&\geq X^T(A(\Gamma^{\prime\prime})-A(\Gamma^{\prime}))X
	\\&=2x_1x_r\geq0.
\end{align*} 
If $\lambda_1(A(\Gamma^{\prime\prime}))=\lambda_1(A(\Gamma^{\prime}))$, then $X$ is also an eigenvector of $A(\Gamma^{\prime\prime})$ corresponding to $\lambda_1(A(\Gamma^{\prime\prime}))$. Based on the following equations,\\

$\lambda_1(A(\Gamma^{\prime}))x_1=\sum\limits_{v_s\in N_{\Gamma^\prime}(v_1)}\sigma^{\prime}(v_sv_1)x_s$,\\

$\lambda_1(A(\Gamma^{\prime}))x_r=\sum\limits_{v_s\in N_{\Gamma^\prime}(v_r)}\sigma^{\prime}(v_sv_r)x_s$,\\

$\lambda_1(A(\Gamma^{\prime\prime}))x_1=\sum\limits_{v_s\in N_{\Gamma^\prime}(v_1)}\sigma^{\prime}(v_sv_1)x_s+x_r$\\
and\\

$\lambda_1(A(\Gamma^{\prime\prime}))x_r=\sum\limits_{v_s\in N_{\Gamma^\prime}(v_r)}\sigma^{\prime}(v_sv_r)x_s+x_1$,\\
we obtain that $x_1=x_r=0$, which contradicts Claim \ref{c1}. Thus, $\lambda_1(A(\Gamma^{\prime\prime}))>\lambda_1(A(\Gamma^{\prime}))$, a contradiction. So, $d_{\Gamma^{\prime}}(v_r)\geq 1$. If $r\geq 3$, then $0=\lambda_1(A(\Gamma^{\prime}))x_r=\sum_{v_i\in N_{\Gamma^{\prime}}(v_r)}x_i>0$, a contradiction. Thus, $r=1$ or $2$. Without loss of generality, assume that $r=1$. Then  $k\geq 2$. Note that
\begin{center}
	$(n-2)x_k\leq\lambda_1(A(\Gamma^{\prime}))x_k=\sum\limits_{v_i\in N_{\Gamma^\prime}(v_k)}x_i\leq d_{\Gamma^{\prime}}(v_k)x_k$,
\end{center}
then  $d_{\Gamma^{\prime}}(v_k) \geq n-2$. If  $d_{\Gamma^{\prime}}(v_k)=n-2$, then each of the $n-2$ entries  of $X$ corresponding to the neighbors of $v_k$ is equal to $x_k$. It implies that $x_2=\cdots=x_n$. If $d_{\Gamma^{\prime}}(v_k)=n-1$, then
\begin{center}
	$(n-2)x_{{k}}\leq\lambda_{1}(A(\Gamma^{\prime}))x_{{k}}=x_{1}+\sum\limits_{v_{i}\in N_{\Gamma^{\prime}}(v_{k})\setminus\{v_{1}\}}x_{i}\leq(d_{\Gamma^{\prime}}(v_{{k}})-1)x_{{k}}=(n-2)x_k$.
\end{center}
Equality forces $x_2=\cdots=x_n$. This means that either $d_{\Gamma^{\prime}}(v_i)=n-2$ or $d_{\Gamma^{\prime}}(v_i)=n-1$ and $v_i$ is adjacent to all other vertices in $V(\Gamma^{\prime})\backslash\{v_1\}$ for any $i\in[2,n]$. As a result, $\Gamma^{\prime}[V(\Gamma^{\prime})\backslash\{v_1\}]\cong(K_{n-1},+)$. Since $\Gamma$ is a $t\mathcal{C}_{3}^{-}$-free unbalanced signed graph that is not switching isomorphic to $\Gamma_{n,t+1}$$(3\leq t\leq n-2)$, then we assert that there exist exactly $t-1$ vertices of degree $n-1$. Otherwise, let $r(r\neq t-1)$ denote the number of vertices of degree $n-1$. If $r>t$, then $\Gamma^{\prime}$ contains $t\mathcal{{C}}_3^-$, a contradiction. If $r=t$, then $\Gamma$ is switching isomorphic to $\Gamma_{n,t+1}$$(3\leq t\leq n-1)$, another contradiction. For the case $r\leq t-2$, assume that $d_{\Gamma^{\prime}}(v_i)=n-1$ for $i\in[2,r+1]$. Let $\Gamma^{\prime\prime}= \Gamma^{\prime}+v_1v_{r+2}$, then $\Gamma^{\prime\prime}$ is a $t\mathcal{C}_{3}^{-}$-free unbalanced signed graph, not switching isomorphic to $\Gamma_{n,t+1}$ and
\begin{align*}
	\lambda_1(A(\Gamma^{\prime\prime}))-\lambda_1(A(\Gamma^{\prime}))&\geq X^T(A(\Gamma^{\prime\prime})-A(\Gamma^{\prime}))X
	\\&=2x_1x_{r+2}\geq0.
\end{align*} 
If $\lambda_1(A(\Gamma^{\prime\prime}))=\lambda_1(A(\Gamma^{\prime}))$, then $X$ is also an eigenvector of $A(\Gamma^{\prime\prime})$ corresponding to $\lambda_1(A(\Gamma^{\prime\prime}))$. From the following equations,\\

$\lambda_1(A(\Gamma^{\prime}))x_1=\sum\limits_{v_s\in N_{\Gamma^\prime}(v_1)}\sigma^{\prime}(v_sv_1)x_s$,\\

$\lambda_1(A(\Gamma^{\prime}))x_{r+2}=\sum\limits_{v_s\in N_{\Gamma^\prime}(v_{r+2})}\sigma^{\prime}(v_sv_{r+2})x_s$,\\

$\lambda_1(A(\Gamma^{\prime\prime}))x_1=\sum\limits_{v_s\in N_{\Gamma^\prime}(v_1)}\sigma^{\prime}(v_sv_1)x_s+x_{r+2}$\\
and\\

$\lambda_1(A(\Gamma^{\prime}))x_{r+2}=\sum\limits_{v_s\in N_{\Gamma^\prime}(v_{r+2})}\sigma^{\prime}(v_sv_{r+2})x_s+x_1$,\\
we obtain that $x_1=x_{r+2}=0$, which contradicts Claim \ref{c1}. Thus,  $\lambda_1(A(\Gamma^{\prime\prime}))>\lambda_1(A(\Gamma^{\prime}))$, a contradiction.
Thus, there exist exactly $t-1$ vertices with a degree of $n-1$ and $\Gamma$ is switching isomorphic to $\Gamma_{n,t}$. For $3\leq t\leq n-3$, it is clear that  $3\leq t\leq\left\lfloor\frac{n}{2}\right\rfloor$, then $\lambda_1(A(\Gamma_{n,t}))>\lambda_1(A(\Sigma_{1,t-1,n-t-2}))$. If $\left\lfloor\frac{n}{2}\right\rfloor+1\leq t\leq n-3$, then $\lambda_1(A(\Sigma_{1,t-1,n-t-2}))>\lambda_1(A(\Gamma_{n,t}))$ by Lemma \ref{lq1}. For $t=n-2$, $\Gamma$ is switching isomorphic to $\Gamma_{n,n-2}$ since $\Gamma$ is not switching isomorphic to $\Gamma_{n,n-1}$. For $t\geq n-1$, a similar reasoning yields $x_2=\cdots=x_n$, additionally, either $d_{\Gamma^{\prime}}(v_i)=n-2$ or $d_{\Gamma^{\prime}}(v_i)=n-1$ and $v_i$ is adjacent to all other vertices $V(\Gamma^{\prime})\backslash\{v_1\}$ for any $i\in[2,n]$. It follows that $\Gamma^{\prime}[V(\Gamma^{\prime})\backslash\{v_1\}]\cong(K_{n-1},+)$. Next, we assert that there exist exactly $n-2$ vertices of degree $n-1$. Suppose for contradiction that there exist $r(r\neq t)$ such vertices.  If $r=n-1$, then $\Gamma$ is switching isomorphic to $\Gamma_{n,n}$$(t\geq n-1)$, a contradiction. If $r\leq n-3$, assume that $d_{\Gamma^{\prime}}(v_i)=n-1$ for $i\in[2,r+1]$. Let $\Gamma^{\prime\prime}= \Gamma^{\prime}+v_1v_{r+2}$, then $\Gamma^{\prime\prime}$ is a $t\mathcal{C}_{3}^{-}$-free unbalanced signed graph, not switching isomorphic to $\Gamma_{n,n}$ and
\begin{align*}
	\lambda_1(A(\Gamma^{\prime\prime}))-\lambda_1(A(\Gamma^{\prime}))&\geq X^T(A(\Gamma^{\prime\prime})-A(\Gamma^{\prime}))X
	\\&=2x_1x_{r+2}\geq0.
\end{align*} 
If $\lambda_1(A(\Gamma^{\prime\prime}))=\lambda_1(A(\Gamma^{\prime}))$, then $X$ is also an eigenvector of $A(\Gamma^{\prime\prime})$ corresponding to $\lambda_1(A(\Gamma^{\prime\prime}))$. Consider the following equations,\\

$\lambda_1(A(\Gamma^{\prime}))x_1=\sum\limits_{v_s\in N_{\Gamma^\prime}(v_1)}\sigma^{\prime}(v_sv_1)x_s$,\\

$\lambda_1(A(\Gamma^{\prime}))x_{r+2}=\sum\limits_{v_s\in N_{\Gamma^\prime}(v_{r+2})}\sigma^{\prime}(v_sv_{r+2})x_s$,\\

$\lambda_1(A(\Gamma^{\prime\prime}))x_1=\sum\limits_{v_s\in N_{\Gamma^\prime}(v_1)}\sigma^{\prime}(v_sv_1)x_s+x_{r+2}$\\
and\\

$\lambda_1(A(\Gamma^{\prime}))x_{r+2}=\sum\limits_{v_s\in N_{\Gamma^\prime}(v_{r+2})}\sigma^{\prime}(v_sv_{r+2})x_s+x_1$,\\
we obtain that $x_1=x_{r+2}=0$, which contradicts Claim \ref{c1}. Thus,  $\lambda_1(A(\Gamma^{\prime\prime}))>\lambda_1(A(\Gamma^{\prime}))$, a contradiction. Hence, $\Gamma^{\prime}$ is switching isomorphic to $\Gamma_{n,n-1}$ for $t\geq n-1$.
\begin{cccase}\label{Case2}
	$X>0$.
\end{cccase}
First, we analyze the case where $3\leq t \leq n-3$. We assert that $\Gamma^{\prime}[V(\Gamma^{\prime})\backslash\{v_1,v_2\}]\cong(K_{n-2},+)$. To prove this by contradiction, suppose there exists an edge $uv\notin \Gamma^{\prime}[V(\Gamma^{\prime})\backslash\{v_1,v_2\}]$, let $\Gamma^{\prime\prime}=\Gamma^{\prime}+uv$, then $\Gamma^{\prime\prime}$ is a $t\mathcal{C}_{3}^{-}$-free unbalanced signed graph, not switching isomorphic to $\Gamma_{n,t+1}$ and
\begin{align*}
	\lambda_1(A(\Gamma^{\prime\prime}))-\lambda_1(A(\Gamma^{\prime}))&\geq X^T(A(\Gamma^{\prime\prime})-A(\Gamma^{\prime}))X
	\\&=2x_ux_{v}>0,
\end{align*} 
a contradiction. Thus, $\Gamma^{\prime}[V(\Gamma^{\prime})\backslash\{v_1,v_2\}]\cong(K_{n-2},+)$. Next, we assert that $|N_{\Gamma^{\prime}}(v_1)\cap N_{\Gamma^{\prime}}(v_2)|=t-1$. Otherwise, $|N_{\Gamma^{\prime}}(v_1)\cap N_{\Gamma^{\prime}}(v_2)|\leq t-2$, which implies the existence of a vertex $v_s\notin {N_{\Gamma^{\prime}}(v_1)\cap N_{\Gamma^{\prime}}(v_2)}$. If $v_s\notin N_{\Gamma^{\prime}}(v_1)$ and $v_s\notin N_{\Gamma^{\prime}}(v_2)$ simultaneously, let $\Gamma^{\prime\prime}= \Gamma^{\prime}+v_1v_s+v_2v_s$, then $\Gamma^{\prime\prime}$ is a $t\mathcal{C}_{3}^{-}$-free unbalanced signed graph, not switching isomorphic to $\Gamma_{n,t+1}$ and
\begin{align*}
	\lambda_1(A(\Gamma^{\prime\prime}))-\lambda_1(A(\Gamma^{\prime}))&\geq X^T(A(\Gamma^{\prime\prime})-A(\Gamma^{\prime}))X
	\\&=2(x_1+x_2)x_s>0,
\end{align*} 
a contradiction. For the cases where $v_s\in {N_{\Gamma^{\prime}}(v_1)\setminus N_{\Gamma^{\prime}}(v_2)}$ or $v_s\in {N_{\Gamma^{\prime}}(v_2)\setminus N_{\Gamma^{\prime}}(v_1)}$, analogous contradictions are uniformly derived in both scenarios. 
 Thus, $|N_{\Gamma^{\prime}}(v_1)\cap N_{\Gamma^{\prime}}(v_2)|=t-1$. Now, we claim that either $v_iv_1\in E(\Gamma^{\prime})$ or $v_iv_2\in E(\Gamma^{\prime})$ for any $v_i\in V(\Gamma^{\prime})\backslash(N_{\Gamma^{\prime}}[v_1]\cap N_{\Gamma^{\prime}}[v_2])$. Assume for contradiction that there exists $v_s$ 
such that $v_sv_1\notin E(\Gamma^{\prime})$ and  $v_sv_2\notin E(\Gamma^{\prime})$,  let  $\Gamma^{\prime\prime}= \Gamma^{\prime}+v_2v_s$, then $\Gamma^{\prime\prime}$ is a $t\mathcal{C}_{3}^{-}$-free unbalanced signed graph, not switching isomorphic to $\Gamma_{n,t+1}$ and
\begin{align*}
	\lambda_1(A(\Gamma^{\prime\prime}))-\lambda_1(A(\Gamma^{\prime}))&\geq X^T(A(\Gamma^{\prime\prime})-A(\Gamma^{\prime}))X
	\\&=2x_2x_{s}>0,
\end{align*} 
a contradiction. To conclude this line of reasoning, we assert that $d_{\Gamma^{\prime}}(v_1)=t+1$. If $d_{\Gamma^{\prime}}(v_1)=t$, then $\Gamma$ is switching isomorphic to $\Gamma_{n,t+1}$$(3\leq t\leq n-3)$.  If $d_{\Gamma^{\prime}}(v_1)\geq t+2$, then $\Gamma^\prime$ is switching isomorphic to $\Sigma_{d_{\Gamma^{\prime}}(v_1)-t,t-1,n-d_{\Gamma^{\prime}}(v_1)-1}$. Let $\Gamma^{\prime\prime}= \Gamma^{\prime}-v_1v_r+v_2v_r$ for  $v_r\in  N_{\Gamma^{\prime}}[v_1]\backslash(N_{\Gamma^{\prime}}[v_1]\cap N_{\Gamma^{\prime}}[v_2])$, then  $\Gamma^{\prime\prime}$ is a $t\mathcal{C}_{3}^{-}$-free unbalanced signed graph, not switching isomorphic to $\Gamma_{n,t+1}$ and
\begin{align*}
	\lambda_1(A(\Gamma^{\prime\prime}))-\lambda_1(A(\Gamma^{\prime}))&\geq X^T(A(\Gamma^{\prime\prime})-A(\Gamma^{\prime}))X
	\\&=2x_r(x_2-x_1)\geq0.
\end{align*} 
If $\lambda_1(A(\Gamma^{\prime\prime}))=\lambda_1(A(\Gamma^{\prime}))$, then $X$ is also an eigenvector of $A(\Gamma^{\prime\prime})$ corresponding to $\lambda_1(A(\Gamma^{\prime\prime}))$. Based on the following equations,\\

$\lambda_1(A(\Gamma^{\prime}))x_1=\sum\limits_{v_s\in N_{\Gamma^{\prime}}(v_1)}x_s$\\
and\\

$\lambda_1(A(\Gamma^{\prime\prime}))x_1=\sum\limits_{v_s\in N_{\Gamma^{\prime}}(v_1)}x_s-x_r$,\\
we obtain that  $x_r=0$, which contradicts Case \ref{Case2}. Thus, $\lambda_1(A(\Gamma^{\prime\prime}))>\lambda_1(A(\Gamma^{\prime}))$, a contradiction. We therefore conclude $d_{\Gamma^{\prime}}(v_1)=t+1$, so $\Gamma^{\prime}$ is switching isomorphic to $\Sigma_{1,t-1,n-t-2}$. Note that for $3\leq t\leq\left\lfloor\frac{n}{2}\right\rfloor$, $\lambda_1(A(\Gamma_{n,t}))>\lambda_1(A(\Sigma_{1,t-1,n-t-2}))$, and for  $\left\lfloor\frac{n}{2}\right\rfloor+1\leq t\leq n-3$, $\lambda_1(A(\Sigma_{1,t-1,n-t-2}))>\lambda_1(A(\Gamma_{n,t}))$ by Lemma \ref{lq1}. For $t=n-2$, we similarly have  $|N_{\Gamma^{\prime}}(v_1)\cap N_{\Gamma^{\prime}}(v_2)|=n-3$. Without loss of generality, assume that $|N_{\Gamma^{\prime}}(v_1)\cap N_{\Gamma^{\prime}}(v_2)|=\{v_3,...,v_{n-1}\}$. Next, we assert that $v_n$ is not adjacent to both $v_1$ and $v_2$. Otherwise, if $v_n$ is adjacent to both $v_1$ and $v_2$, then $\Gamma^{\prime}\cong \Gamma_{n,n}$, and  $\Gamma^{\prime}$ contains $(n-2)\mathcal{C}_{3}^{-}$, a contradiction. If $v_n$ is adjacent to either $v_1$ or $v_2$, then $\Gamma^{\prime}\cong \Gamma_{n,n-1}$, another contradiction. Thus, $v_n$ is not adjacent to both $v_1$ and $v_2$.  Hence, $\Gamma$ is switching isomorphic to $U_1$. However, $\lambda_1(A(\Gamma_{n,n-2}))>\lambda_1(A(U_1)$ by Lemma \ref{lqq1}. Thus,  $\Gamma$ is switching isomorphic to $\Gamma_{n,n-2}$ for $t=n-2$. For $n-1 \leq t$, a parallel argument gives $\Gamma^{\prime}[V(\Gamma^{\prime})\backslash\{v_1,v_2\}]\cong(K_{n-2},+)$. Next, we assert that $d_{\Gamma^{\prime}}(v_2)=n-1$. Suppose for contradiction $d_{\Gamma^{\prime}}(v_2)\leq n-2$. Without loss of generality, assume that $v_2v_r\notin E(\Gamma^{\prime})$, let $\Gamma^{\prime\prime}= \Gamma^{\prime}+v_2v_r$, then $\Gamma^{\prime\prime}$ is a $t\mathcal{C}_{3}^{-}$-free unbalanced signed graph, not switching isomorphic to $\Gamma_{n,n}$ and
\begin{align*}
	\lambda_1(A(\Gamma^{\prime\prime}))-\lambda_1(A(\Gamma^{\prime}))&\geq X^T(A(\Gamma^{\prime\prime})-A(\Gamma^{\prime}))X
	\\&=2x_2x_r>0,
\end{align*} 
a contradiction. Thus, $d_{\Gamma^{\prime}}(v_2)=n-1$. Finally, we claim that $d_{\Gamma^{\prime}}(v_1)=n-2$. If $d_{\Gamma^{\prime}}(v_1)=n-1$, then $\Gamma^{\prime}$ is switching isomorphic to $\Gamma_{n,n}$, a contradiction. If $d_{\Gamma^{\prime}}(v_1)\leq n-3$, without loss of generality, assume that $v_1v_s\notin E(\Gamma^{\prime})$, let $\Gamma^{\prime\prime}= \Gamma^{\prime}+v_1v_s$, then $\Gamma^{\prime\prime}$ is a $t\mathcal{C}_{3}^{-}$-free unbalanced signed graph, not switching isomorphic to $\Gamma_{n,n}$ and
\begin{align*}
	\lambda_1(A(\Gamma^{\prime\prime}))-\lambda_1(A(\Gamma^{\prime}))&\geq X^T(A(\Gamma^{\prime\prime})-A(\Gamma^{\prime}))X
	\\&=2x_1x_s>0,
\end{align*} 
a contradiction. Thus, $d_{\Gamma^{\prime}}(v_1)=n-2$ and $\Gamma^{\prime}$ is switching isomorphic to $\Gamma_{n,n-1}$. The proof is completed.

\end{document}